\documentclass[11pt]{amsart}
\usepackage{amssymb}

\newtheorem{theo}{Theorem} % Pour les theoremes principaux
\newtheorem{exttheo}{Theorem} % Pour les theoremes exterieurs

\newtheorem{lemma}{Lemma}[section]
\newtheorem{prop}[lemma]{Proposition}

\newtheorem{claim}[lemma]{Claim}

\theoremstyle{remark}
\newtheorem{remark}[lemma]{Remark}

\theoremstyle{definition}

\newcommand{\RR}{\mathbb{R}}

\newcommand{\eps}{\varepsilon}

\newcommand{\BB}{B}

\newcommand{\EEE}{\mathcal{E}}

\newcommand{\LLL}{\mathcal{L}}

\newcommand{\III}{\mathcal{I}}
\newcommand{\tIII}{\widetilde{\mathcal{I}}}

\newcommand{\YYY}{\mathcal{Y}}
\newcommand{\tYYY}{\widetilde{\mathcal{Y}}}

\newcommand{\tS}{\tilde{S}}

\newcommand{\tQ}{\widetilde{Q}}

\newcommand{\tW}{\widetilde{W}}

\newcommand{\tG}{\tilde{G}}

\newcommand{\tomega}{\widetilde{\omega}}

\DeclareMathOperator{\re}{Re}
\DeclareMathOperator{\im}{Im}

\newcommand{\hdot}{\dot{H}^1}
\newcommand{\Hdot}{\dot{H}^1(\RR^N)}

\newcommand{\EMPH}[1]{\medskip\noindent\emph{#1}.}

\numberwithin{equation}{section} %pour numeroter les equations par section

\title[Scattering norm estimate for critical wave equation]{Scattering norm estimate near the threshold for energy-critical focusing semilinear wave equation}
\author[T.~Duyckaerts]{Thomas Duyckaerts$^1$}
\email{thomas.duyckaerts@u-cergy.fr}
\address{Thomas Duyckaerts\\
Universit{\'e} de Cergy-Pontoise\\
D\'epartement de Math\'ematiques\\ 
Site de Saint Martin, 2 avenue Adolphe-Chauvin\\ 
95302 Cergy-Pontoise cedex, France. }
\author[F.~Merle]{Frank Merle$^2$}
\thanks{$^1$Cergy-Pontoise (UMR 8088)}
\thanks{$^2$Cergy-Pontoise, IHES, CNRS}
\thanks{This work was partially supported by the French ANR Grant ONDNONLIN}

\date{\today}

\begin{document}

\begin{abstract}
We consider the energy-critical semilinear focusing wave equation in dimension $N=3,4,5$. An explicit solution $W$ of this equation is known. By the work of C.~Kenig and F.~Merle, any solution of initial condition $(u_0,u_1)$ such that $E(u_0,u_1)<E(W,0)$ and $\|\nabla u_0\|_{L^2}<\|\nabla W\|_{L^2}$ is defined globally and has finite $L^{\frac{2(N+1)}{N-2}}_{t,x}$-norm, which implies that it scatters. In this note, we show that the supremum of the $L^{\frac{2(N+1)}{N-2}}_{t,x}$-norm taken on all scattering solutions at a certain level of energy below $E(W,0)$ blows-up logarithmically as this level approaches the critical value $E(W,0)$. We also give a similar result in the case of the radial energy-critical focusing semilinear Schr\"odinger equation. The proofs rely on the compactness argument of C.~Kenig and F.~Merle, on a classification result, due to the authors, at the energy level $E(W,0)$, and on the analysis of the linearized equation around $W$.
\end{abstract}

\maketitle

%\tableofcontents

\section{Introduction}
We consider the focusing energy-critical wave equation on an interval $I$ ($0\in I$)
\begin{equation}
\label{CP}
\left\{ 
\begin{gathered}
\partial_t^2 u -\Delta u-|u|^{\frac{4}{N-2}}u=0,\quad (t,x)\in I\times \RR^N\\
u_{\restriction t=0}=u_0\in \hdot,\quad \partial_t u_{\restriction t=0}=u_1\in L^2,
\end{gathered}\right.
\end{equation}
where $u$ is real-valued, $N\in\{3,4,5\}$, $L^2:=L^2(\RR^N)$ and $\hdot:=\Hdot$. The equation \eqref{CP} is locally well-posed in $\hdot\times L^2$ (see \cite{Pecher84}, \cite{GiSoVe92} and \cite{ShSt94}): if $(u_0,u_1)\in \hdot\times L^2$, there exists an unique solution $u$, defined on a maximal time of existence $I_{\max}$ and such that for all interval $J$
$$J\Subset I_{\max}\Longrightarrow \|u\|_{S(J)}<\infty,\text{ where }S(J):=L^{\frac{2(N+1)}{N-2}}\left(J\times \RR^N\right).$$
Furthermore, the solution $u$ of \eqref{CP} scatters forward in time in $\hdot\times L^2$ if and only $$ [0,+\infty) \subset I_{\max}\text{ and }\|u\|_{S(0,+\infty)}<\infty.$$
Thus the norm $S(\RR)$ measures the nonlinear effect for a given solution. The energy
$$ E(u(t),\partial_tu(t))=\frac{1}{2} \int |\partial_t u(t,x)|^2dx+\frac{1}{2} \int |\nabla u(t,x)|^2dx-\frac{N-2}{2N}|u(t,x)|^{\frac{2N}{N-2}}dx$$
is conserved for solutions of \eqref{CP}. 

The \emph{defocusing} case (equation \eqref{CP} with sign $+$ instead of $-$ in front of the nonlinearity) has been the object of intensive studies in the last decades (see for example \cite{ShSt98} and references therein). In this case the solutions are known to scatter, which implies, for any solution $u$, a bound of the norm $S(\RR)$ by an unspecified function of the defocusing energy 
$$E_{d}=\frac{1}{2} \int |\partial_t u|^2+\frac{1}{2} \int |\nabla u|^2+\frac{N-2}{2N}\int |u|^{\frac{2N}{N-2}}.$$
 In three spatial dimension, an explicit upper bound was proven by T.~Tao \cite{Tao06DPDE}: for any solution $u$ of the defocusing equation,
$$ \|u\|_{L^{4}_tL^{12}_x}\leq C(1+E_d)^{CE_d^{105/2}},$$
which gives, by Strichartz and interpolation estimate, a similar bound for $\|u\|_{S(\RR)}$.

Going back to the \emph{focusing} case,  consider the explicit $\hdot$ stationnary solution of \eqref{CP}
\begin{equation}
\label{defW}
W:=\frac{1}{\left(1+\frac{|x|^2}{N(N-2)}\right)^{\frac{N-2}{2}}}.
\end{equation}
In \cite{KeMe06Pb}, C.~Kenig and F.~Merle have described the dynamics of \eqref{CP} below the energy threshold $E(W,0)$. Namely, if $E(u_0,u_1)<E(W,0)$, then $\int |\nabla u_0|^2\neq \int |\nabla W|^2$ and the solution $u$ scatters (both forward and backward in time) if and only if $\int |\nabla u_0|^2<\int |\nabla W|^2$. %If $\int |\nabla u_0|^2>\int |\nabla W|^2$, $u$ has a bounded interval of existence. 
This implies that for $\eps>0$ the following supremum is finite:
$$ \III_{\eps}=\sup_{u\in F_{\eps}} \int_{\RR\times \RR^N}|u(t,x)|^{\frac{2(N+1)}{N-2}}\,dtdx=\sup_{u\in F_{\eps}} \|u\|_{S(\RR)}^{\frac{2(N+1)}{N-2}},$$
where 
$$F_{\eps}:=\left\{u \text{ solution of }\eqref{CP}\text{ such that }E(u_0,u_1)\leq E(W,0)-\eps^2 \text{ and }\int |\nabla u_0|^2<\int |\nabla W|^2\right\}.$$  
Furthermore, the existence of the non-scattering solution $W$ at the energy threshold shows that
\begin{equation*}
\lim_{\eps\rightarrow 0^+} \III_{\eps}=+\infty.
\end{equation*}

The purpose of this note is to give an equivalent of $\III_{\eps}$ for small $\eps$. Consider the negative eigenvalue $-\omega^2$ ($\omega>0$) of the linearized operator associated to \eqref{CP} around $W$:
$$-\omega^2=\inf_{\substack{u\in H^1\\ \int u^2=1}} \int_{\RR^N}|\nabla u|^2-\frac{N+2}{N-2}\int_{\RR^N} W^{\frac{4}{N-2}}|u|^2.$$
(See \S \ref{SS:lin} for details).  Then
\begin{theo}
\label{maintheo}
$$\lim_{\eps\rightarrow 0^+} \frac{\III_{\eps}} {|\log \eps|}=\frac{2}{\omega}\int_{\RR^N} W^{\frac{2(N+1)}{N-2}}.$$
\end{theo}

\begin{remark}
It would be interesting to get an explicit value of the limit 
$\frac{2}{\omega}\int_{\RR^N} W^{\frac{2(N+1)}{N-2}}.$
A straightforward computation gives:
\begin{align*}
\int_{\RR^N} W^{\frac{2(N+1)}{N-2}}&=\frac{(N(N-2))^{\frac{N}{2}}}{2^{2N+1}}\times \frac{N!}{\left((\frac N2)!\right)^2}\times \pi&&\text{ if }N\text{ is even,}\\
\int_{\RR^N} W^{\frac{2(N+1)}{N-2}}&=\frac{(N(N-2))^{\frac{N}{2}}}{2}\times \frac{\left(\frac{N-1}{2}\right)!}{N!}&&\text{ if }N\text{ is odd}.
\end{align*}
However we do not know any explicit expression of $\omega$.
\end{remark}

Let us give an outline of the proof of Theorem \ref{maintheo}. In Section \ref{S:compactness}, we show that a sequence of solutions $(u_n)$ such that 
$$E(u_n(0),\partial_t u_n(0))<E(W,0), \quad \int|\nabla u_n(0)|^2<\int |\nabla W|^2 \text{ and } \lim_{n\rightarrow +\infty}\|u_n\|_{S(\RR)}=+\infty$$
must converge to $W$ up to modulation for a well-chosen time sequence.
This relies on the compactness argument of \cite[Section 4]{KeMe06Pb}, using the profile decomposition of Bahouri-G\'erard \cite{BaGe99}, and on the classification of the solutions of \eqref{CP} at the threshold of energy in our previous work \cite{DuMe07b}. The second step of the proof is an analysis of the behaviour of solutions whose initial conditions are close to $(W,0)$, which is carried out in Section \ref{S:estimates}. We show, as a consequence of the existence of the negative eigenvalue $-\omega^2$, that such solutions go away from the solution $W$ in a time which is of logarithmic order with respect to the distance of the initial condition to $(W,0)$. In Section \ref{S:proof} we put together the preceding arguments to prove Theorem \ref{maintheo}.

Our arguments do not depend strongly on the nature of equation \eqref{CP}, and we except that a logarithmic estimate of the scattering norm $S(\RR)$ near the threshold holds in similar situations, as long as the linearized operator around the ground state admits real nonzero eigenvalues. In Section \ref{S:NLS} we give a result and a sketch of proof in the case of the radial, energy-critical focusing nonlinear Schr\"odinger equation.

\section{Convergence to $W$ and $W^-$ near the threshold}
\label{S:compactness}

In all the article, we will denote by $\|\cdot\|_p$ the $L^p$ norm on $\RR^N$.

Equation \eqref{CP} enjoys the following invariances: if $u$ is a solution and $t_0\in\RR$, $x_0\in \RR^N$, $\lambda_0>0$, $\delta_0,\delta_1\in \{-1,+1\}$, then
$$ v(t,x)=\frac{\delta_0}{\lambda_0^{(N-2)/2}}u\Big(\frac{t_0+\delta_1 t}{\lambda_0},\frac{x+x_0}{\lambda_0}\Big)$$
is also a solution. Note that the energy of $u$ and, if $u$ is globally defined, the norm $\|u\|_{S(\RR)}$ are not changed by these transformations.

We recall the following classification Theorem, proven in \cite{KeMe06Pb}, for the case $E(u_0,u_1)<E(W,0)$, and in \cite{DuMe07b} for the existence of $W^-$ and the case $E(u_0,u_1)=E(W,0)$:
\begin{exttheo}[Kenig-Merle,Duyckaerts-Merle]
\label{T:classification}
There exists a global solution $W^-$ of \eqref{CP} such that 
\begin{gather*}
E(W^-(0),\partial_tW^-(0))=E(W,0), \quad \|\nabla W_-(0)\|_2<\|\nabla W\|_2\\
\|W^-\|_{S(-\infty,0)}<\infty,\quad \lim_{t\rightarrow +\infty} \|\nabla(W^-(t)-W)\|_2+\|\partial_t(W^-(t)-W)\|_2=0.
\end{gather*}
Moreover, if $u$ is a solution of \eqref{CP} such that $E(u_0,u_1)\leq E(W,0)$ and $\|\nabla u_0\|_2\leq \|\nabla W\|_2$, then $u$ is globally defined. If furthermore $\|u\|_{S(\RR)}=\infty$, then $u=W^-$ or $u=W$ up to the invariances of the equation.
\end{exttheo}
We will also need the following simple version of long-time perturbation theory results (see e.g. \cite[Theorem 2.20]{KeMe06Pb}).
\begin{lemma}
\label{L:LTPT}
Let $M>0$. Then there exist positive constants $\eps(M)$ and $C(M)$ such that for all solutions $v$ and $u$ of \eqref{CP}, with initial conditions $(v_0,v_1)$ and $(u_0,u_1)$, if the forward time of existence of $v$ is infinite and
\begin{equation*}
%\label{tildeu_scatters}
\|v\|_{S(0,+\infty)}\leq M\text{ and }\|\nabla(u_0-v_0)\|_{2}+\|u_1-v_1\|_{2}\leq \eps(M),
\end{equation*}  
then $u$ is globally defined for positive times and $\|u\|_{S(0,+\infty)}\leq C(M)$. A similar statement holds for negative times.
\end{lemma}
In this section we show the following:
\begin{prop}
\label{P:toW}
Let $u_n$ be a family of solutions of \eqref{CP}, such that
\begin{equation}
\label{below_threshold}
E\big(u_n(0),\partial_tu_n(0)\big)<E(W,0),\quad \|\nabla u_n(0)\|_2<\|\nabla W\|_2.
\end{equation} 
and $\lim_{n\rightarrow +\infty} \|u_n\|_{S(\RR)}=+\infty$. Let
$(t_{n})_n$ be a time sequence.
\begin{enumerate}
\item \label{C:toW} Assume
$$\lim_{n\rightarrow +\infty} \|u_n\|_{S(-\infty,t_n)}=\lim_{n\rightarrow +\infty} \|u_{n}\|_{S(t_n,+\infty)}=+\infty.$$
Then, up to the extraction of a subsequence there exist $\delta_0\in \{-1,+1\}$ and sequences of parameters $x_n\in \RR^n$, $\lambda_n>0$ such that
\begin{equation*}
%\label{uepstoW}
\lim_{n\rightarrow +\infty} 
\left\|\frac{\delta_0}{\lambda_n^{N/2}}\nabla u_{n}\left(t_{n},\frac{\cdot-x_n}{\lambda_n}\right)-\nabla  W\right\|_{2}+\left\|\frac{\partial u_{n}}{\partial t}\left(t_{n}\right)\right\|_{2}=0.
\end{equation*}
\item \label{C:toW'} Assume that there exists $C_0\in (0,+\infty)$ such that
$$\lim_{n\rightarrow +\infty} \|u_n\|_{S(-\infty,t_n)}=+\infty \text{ and }\lim_{n\rightarrow +\infty} \|u_{n}\|_{S(t_n,+\infty)}=C_0.$$
Then, up to the extraction of a subsequence there exist $t_0\in \RR$, $\delta_0,\delta_1\in\{-1,+1\}$, and sequences of parameters $x_n\in \RR^n$, $\lambda_n>0$ such that
\begin{multline*}
%\label{uepstoW-}
\qquad \qquad \lim_{n\rightarrow +\infty} \left\|\frac{\delta_0}{\lambda_n^{N/2}}\nabla u_{n}\left(t_{n},\frac{\cdot-x_n}{\lambda_n}\right)-\nabla  W^-(t_0)\right\|_{2}
\\+\left\|\frac{\delta_0}{\lambda_n^{N/2}}\frac{\partial u_{n}}{\partial t}\left(t_{n},\frac{\cdot-x_n}{\lambda_n}\right)-\frac{\partial W^-}{\partial t}(t_0)\right\|_{2}=0.\qquad\qquad
\end{multline*}
\end{enumerate}
\end{prop}
\begin{remark}
Case \eqref{C:toW'} will not be used in the proof of Theorem \ref{maintheo}, and is stated only for its own interest.
\end{remark}
\begin{proof}[Sketch of Proof]
We will sketch the proof \eqref{C:toW}, the proof of \eqref{C:toW'} is similar and left to the reader. Translating in time all the $u_{n}$, we may assume that $t_n=0$ for all $n$, and thus
\begin{equation}
\label{limS+}
\lim_{n\rightarrow+\infty}\|u_{n}\|_{S(-\infty,0)}=\lim_{n\rightarrow+\infty}\|u_{n}\|_{S(0,+\infty)}=+\infty.
\end{equation}

In view of \eqref{below_threshold} and \eqref{limS+}, one can show, using the profile decomposition of \cite{BaGe99} as in \cite[Proposition 4.2]{KeMe06Pb}, that there exist (up to the extraction of a subsequence) parameters $\lambda_n>0$, $x_n\in \RR^N$, and functions $(v_0,v_1)\in \hdot\times L^2$ such that
$$
\lim_{n\rightarrow +\infty} \left\|\frac{1}{\lambda_n^{N/2}}\nabla u_{n}\left(0,\frac{\cdot-x_n}{\lambda_n}\right)-\nabla  v_0\right\|_{2}
\\+\left\|\frac{1}{\lambda_n^{N/2}}\frac{\partial u_{n}}{\partial t}\left(0,\frac{\cdot-x_n}{\lambda_n}\right)-v_1\right\|_{2}=0.
$$
We refer to \cite[Section 4]{KeMe06} and also \cite[Lemma 2.5]{DuMe07a} for proofs in the case of nonlinear Schr\"odinger equations that readily apply to our case. Note that 
\begin{equation}
\label{condv1}
\|\nabla v_0\|_{2}\leq \|\nabla W\|_2,\quad E(v_0,v_1)\leq E(W,0).
\end{equation}
Let $v$ be the solution of \eqref{CP} with initial conditions $(v_0,v_1)$. Theorem \ref{T:classification} and \eqref{condv1} imply that $v$ is globally defined. By Lemma \ref{L:LTPT} and by \eqref{limS+}, 
\begin{equation*}
\|v\|_{S(-\infty,0)}=+\infty,\quad \|v\|_{S(0,+\infty)}=+\infty.
\end{equation*}
This shows, again by Theorem \ref{T:classification}, that $v=W$, up to the invariances of equation \eqref{CP} concluding the proof.
\end{proof}

\section{Estimates near the threshold}
\label{S:estimates}
\subsection{Preliminaries on the linearized equation}
\label{SS:lin}
In this subsection, we recall results on the linearized equation near $W$. We refer to \cite{DuMe07b} for the details. Let $u$ be a solution of \eqref{CP} which is close to $W$. Write $u=W+h$. Then $h$ is solution to the equation
\begin{gather}
\label{linearized}
(\partial_t^2+L)h=R(h).\\
\notag
L:=-\Delta- \frac{N+2}{N-2}W^{\frac{4}{N-2}},\quad 
R(h):=|W+h|^{\frac{4}{N-2}}(W+h)-W^{\frac{N+2}{N-2}}-\frac{N+2}{N-2}W^{\frac{4}{N-2}} h.
\end{gather} 
Let
\begin{equation}
\label{def_Wj}
W_0=a\left(\frac{N-2}{2}W+x\cdot \nabla W\right), \quad W_{j}=b \partial_{x_j}W,\; j=1\ldots N.
\end{equation} 
where the constants $a$ and $b$ are chosen so that $\|\nabla W_j\|_2=1$ for $j=0,1,\ldots, N$. By the invariances of equation \eqref{CP}, $L(W_j)=0$ for $j=0,\ldots, N$. As a consequence the functions $W_j$, $j=0,\ldots,N$ are in the kernel of the quadratic form
$$ Q(h)=\frac 12 \int Lh\,h=\frac 12 \int |\nabla h|^2-\frac{N+2}{2(N-2)}\int W^{\frac{4}{N-2}}h^2. $$
Observe that if $(h,\partial_t h)$ is small in $\hdot\times L^2$,
\begin{equation}
\label{dev.E}
E(W+h,\partial_t h)=E(W,0)+Q(h)+\frac 12 \int |\partial_t h|^2+O\left(\|h\|_{\frac{2N}{N-2}}^3\right).
\end{equation} 
Furthermore, one can check that the infimum of $Q(h)$ for $h\in H^1$, $\|h\|_{L^2}=1$, is negative, and thus that $L$ admits a positive, radial eigenfunction $\YYY$ with eigenvalue $-\omega^2<0$. We normalize $\YYY$ such that $\int \YYY^2=1$. The self-adjointness of $L$ implies
\begin{equation}
\label{Y_Wj_ortho}
\int \YYY W_j=0,\quad j=0\ldots N.
\end{equation} 
Consider
$$G_{\bot}:=\left\{ f\in \hdot, \;\int \YYY f=\int \nabla W_0\cdot \nabla f=\ldots=\int \nabla W_N\cdot \nabla f=0\right\}.$$
The following result (see Proposition 5.5 of \cite{DuMe07b}) shows in particular that $-\omega^2$ is the only negative eigenvalue of $L$:
\begin{claim}
\label{C:coercivity}
There exists a constant $c_Q>0$ such that
$$ \forall h\in G_{\bot},\quad Q(h)\geq c_Q\|\nabla h\|^2_{L^2}.$$
\end{claim}
% As a consequence of the preceding lemma, we get
% \begin{claim}
% \label{C:equiv}
% There exists $\delta_0,C>0$ such that if 
% $$ \|\nabla h\|_2+\|\partial_t h\|_2\leq \delta_0,\text{ and }\int h W_j=0,\quad j=0\ldots N,$$
% then
% $$ C^{-1}\left|\|\nabla W\|_2^2-\|\nabla h\|^2\right|\leq \|\nabla h\|_2+\|\partial_t h\|_2 \leq C\left|\|\nabla W\|_2^2-\|\nabla h\|^2\right|.$$
% \end{claim}
As a consequence of the Strichartz estimates for the linear wave equation (see \cite{GiVe95} and \cite{LiSo95}), we easily get the following Strichartz-type estimate for equation \eqref{linearized}.
\begin{claim}
\label{L:strichartz}
There exist constants $\tilde{c},C>0$ such that if $h$ is a solution of \eqref{linearized} on an interval $[t_0,t_1]$ such that 
$$\|\nabla h\|_2+\|\partial_t h\|_2+|t_0-t_1|\leq \tilde{c}.$$
Then 
$$ \|h\|_{S(t_0,t_1)}\leq C\left(\|\nabla h(t_0)\|_2+\|\partial_t h(t_0)\|_2\right).$$
\end{claim}

\subsection{Estimate on the exit time}
\label{SS:bootstrap}
In this subsection, we consider a sequence $u_n$ of solutions of \eqref{CP} such that 
\begin{gather}
\label{hyp.un1}
\lim_{n\rightarrow +\infty} \|\nabla(u_n(0)-W)\|_2+\|\partial_t u_n(0)\|_2=0\\
\label{hyp.un2}
E(W,0)-E(u_n,\partial_tu_n)=\eps_n^2\underset{n\rightarrow +\infty}{\longrightarrow}0\text{ and }\forall n>0,\;\|\nabla u_n(0)\|_2<\|\nabla W\|_2.
\end{gather} 
Let $h_n=u_n-W$ and decompose $h_n$ as
\begin{equation}
\label{devh}
h_n(t)=\beta_n(t)\YYY+\sum_{j=0}^N \gamma_{j,n}(t)W_j +g_n(t),\quad g_n(t)\in G_{\bot}.
\end{equation}
For this, observe that the condition $g_n(t)\in G_{\bot}$ is equivalent to
\begin{gather}
\label{beta}
\beta_n(t)=\int h_n(t)\YYY,\\
\label{gamma}
\gamma_{0,n}(t)=\int \nabla \left(h_n(t)-\beta_n(t)\YYY\right)\cdot \nabla W_0,\quad \gamma_{j,n}(t)=\int \nabla h_n(t)\cdot \nabla W_j,\; j=1,\ldots,N.
\end{gather}
(we used that $\YYY$ and $W$ being radial, $\int\nabla \YYY\cdot\nabla W_j=b\int \nabla \YYY\cdot\partial_{x_j}\nabla W=0$ if $j\in \{1,\ldots, N\}$ and by a similar argument, $\int \nabla W_j\cdot \nabla W_k=0$ if $j\neq k$). We have:
\begin{claim}
\label{C:orthogonality}
Assume \eqref{hyp.un1}. Then there exists sequences $\lambda_n\in (0,+\infty)$, $x_n\in \RR^N$ such that
\begin{equation*}
\lim_{n\rightarrow +\infty}\lambda_n =1,\quad \lim_{n\rightarrow +\infty}x_n=0,
\end{equation*}
and for all $n$, noting $c_0=\int \nabla \YYY \cdot\nabla W_0$,
\begin{gather}
\label{gamma00}
\frac{1}{\lambda_n^{\frac{N}{2}}}\int \nabla u_n\left(\frac{x-x_n}{\lambda_n}\right)\cdot\nabla W_0(x)dx
-\frac{c_0}{\lambda_n^{\frac{N-2}{2}}}\int \left(u_n\left(\frac{x-x_n}{\lambda_n}\right)-W(x)\right)\YYY(x)dx=0\\
\label{gammaj0}
\forall j\in\{1,\ldots,N\},\quad \frac{1}{\lambda_n^{\frac{N}{2}}}\int \nabla u_n\left(\frac{x-x_n}{\lambda_n}\right)\cdot\nabla W_j(x)=0.
\end{gather}
\end{claim}
\begin{proof}[Sketch of proof]
Consider the mapping $J:(\lambda,X,u)\mapsto (J_0,J_1,\ldots, J_N)$ where
\begin{align*}
J_0&=\frac{1}{\lambda^{\frac{N}{2}}}\int \nabla u\left(\frac{x-X}{\lambda}\right)\cdot\nabla W_0(x)dx
-\frac{c_0}{\lambda^{\frac{N-2}{2}}}\int \left(u\left(\frac{x-X}{\lambda}\right)-W(x)\right)\YYY(x)dx\\
J_k&=\frac{1}{\lambda^{\frac{N}{2}}}\int \nabla u\left(\frac{x-X}{\lambda}\right)\cdot\nabla W_k(x),\quad k\in \{1,\ldots,N\}.
\end{align*} 
A straightforward computation shows that $J=0$ and $\left(\frac{\partial J}{\partial\lambda},\frac{\partial J}{\partial X_1},\ldots,\frac{\partial J}{\partial X_N}\right)$ is diagonal and invertible at the point $(1,0,\ldots,0,W)$. The Claim then follows from \eqref{hyp.un1} and the implicit function theorem.
\end{proof}
Observe that the conditions \eqref{gamma00} and \eqref{gammaj0} are equivalent, by \eqref{beta} and \eqref{gamma}, to the condition that the parameters $\gamma_{jn}$ corresponding to the modulated solution $\frac{1}{\lambda_n^{\frac{N-2}{2}}}u_n\left(\frac{t}{\lambda_n},\frac{x}{\lambda_n}\right)$ vanish at $t=0$. By Claim \ref{C:orthogonality} we can assume, up to translation and scaling, that 
\begin{equation}
\label{orthogonality1}
\forall n,\;\forall j\in \{0,1,\ldots,N\},\quad \gamma_{j,n}(0)=0.
\end{equation} 
The main result of this subsection is the following:
\begin{prop}
\label{P:exittime}
There exist a constant $\eta_0$, such that for all $\eta\in (0,\eta_0)$, for all sequence $(u_n)$ satisfying \eqref{hyp.un1}, \eqref{hyp.un2}, \eqref{orthogonality1} and such that $\beta_n(0)\beta_n'(0)\geq 0$, if
$$ T_n(\eta)=\inf\big\{t\geq 0 \;:\;\,|\beta_n(t)|\geq \eta\big\},$$
then for large $n$, $\beta_n(0)\neq 0$, $T_n(\eta)\in (0,+\infty)$ and
\begin{equation}
\label{estimate_time}
\lim_{n\rightarrow +\infty} \frac{T_n(\eta)}{\big|\log |\beta_n(0)|\big|}=\frac{1}{\omega}.
\end{equation}
Furthermore,
\begin{equation}
\label{estimate_beta'}
\liminf_{n\rightarrow +\infty} |\beta'_n(T_n(\eta))|\geq \omega\eta.
\end{equation} 
\end{prop}
\begin{remark}
If $\beta_n(0)\beta_n'(0)<0$, we may achieve the condition $\beta_n(0)\beta_n'(0)\geq 0$ by considering the solution $u_n(-t,x)$ instead of $u_n(t,x)$.
\end{remark}

The remainder of this subsection is devoted to the proof of Proposition \ref{P:exittime}. We first give, as a consequence of the orthogonality conditions \eqref{orthogonality1}, a purely variational lower bound on $|\beta_n(0)|$ (Claim \ref{C:bound.eps}). We then give (Lemma \ref{L:bootstrap*}) precise estimates on $\beta_n(t)$ and $\|\partial_t h_n(t)\|_2+\|\nabla h_n(t)\|_2$, on an interval $(0,t_n)$ where a priori bounds are assumed. These estimates will give the desired bounds on the exit time $T_n(\eta)$.
We will write:
$$ \|\partial_{t,x}h_n(t)\|_2=\|\nabla h_n(t)\|_2+\|\partial_t h_n(t)\|_2.$$

\begin{claim}
\label{C:bound.eps}
There exists $M_0>0$ such that for all sequence $(u_n)$ of solutions of \eqref{CP} satisfying \eqref{hyp.un1}, \eqref{hyp.un2} and \eqref{orthogonality1} we have
$$\beta_n(0)\neq 0\text{ and }\limsup_{n\rightarrow +\infty} \frac{\|\partial_{t,x}h_n(0)\|_2+\eps_n}{\left|\beta_n(0)\right|}\leq M_0.
$$
\end{claim}
\begin{proof}
Developping the energy as in \eqref{dev.E}, we get
$$ E(W,0)-\eps_n^2=E(W+h_n,\partial_t h_n)=E(W,0)+Q(h_n)+\frac{1}{2}\|\partial_t h_n\|^2_2+O\left(\|\nabla h_n\|^3_2\right).$$
The expression \eqref{devh} of $h$ at $t=0$ yields, in view of \eqref{orthogonality1}
$$\|\nabla h_n(0)\|_2\leq C(|\beta_n(0)|+\|\nabla g_n(0)\|_2).$$
Furthermore, taking into account that $Q(\YYY)<0$ and that the functions $W_j$ are in the kernel of $Q$ for $j=0\ldots N$, we get
$$Q(h_n(0))=-\beta_n^2(0) |Q(\YYY)|+Q(g_n(0)).$$
Combining the preceding estimates, we obtain
$$ \beta_n^2(0)|Q(\YYY)|=\eps_n^2+Q(g_n(0))+\frac 12 \|\partial_t h_n(0)\|_2^2+O\left(\beta_n^3(0)+\|\nabla g_n(0)\|_2^3\right).$$
By Claim \ref{C:coercivity}, $Q(g_n(0))\geq c_Q\|\nabla g_n(0)\|_2^2$. This yields for large $n$,
$$ 2\beta_n^2(0) \left|Q(\YYY) \right|\geq \eps_n^2+c_Q\|\nabla g_n(0)\|_2^2+\frac 12\|\partial_th_n(0)\|_2^2\geq \eps_n^2+c\|\nabla h_n(0)\|_2^2+\frac 12 \|\partial_t h_n(0)\|_2^2,$$
which concludes the proof of the claim.
\end{proof}
Our next result is the following Lemma:
\begin{lemma}[Growth on $\lbrack 0,T_n \rbrack$]
\label{L:bootstrap*}
Let us fix $\omega^+$ and $\omega^-$, close to $\omega$, such that $\omega^-<\omega<\omega^+$.
There exist positive constants $\tau_0$, $K_0$ (depending only on the choice of $\omega_{\pm}$) with the following property. Let $(u_n)_n$ be a sequence of solutions of \eqref{CP} satisfying \eqref{hyp.un1}, \eqref{hyp.un2} and such that $\beta_n(0){\beta_n}'(0)\geq 0$.
Let $M> M_0$ (where $M_0$ is given by Claim \ref{C:bound.eps}). Let $\eta$ such that
\begin{equation}
\label{bound_eta}
0<\eta<\frac{1}{K_0 M^3}.
\end{equation}
Define
\begin{equation}
\label{def_tn}
t_n=t_n(M,\eta)=\inf\big\{t\geq 0 \;:\;\|\partial_{t,x}h_n(t)\|_2\geq M|\beta_n(t)|\text{ or }|\beta_n(t)|\geq \eta\big\}.
\end{equation}
Then there exists $\tilde{n}>0$ such that for $n\geq \tilde{n}$,
\begin{align}
\label{growth_beta}
\forall t\in \left[\tau_0,t_n\right), &\quad \omega^-|\beta_n(t)|\leq |\beta'_n(t)|\leq \omega^+|\beta_n(t)|\\
\label{bootstrap_beta}
\forall t\in \left[\tau_0,t_n\right), &\quad \frac{1}{K_0}\left|\beta_n(0)\right|e^{\omega^- t}\leq \left|\beta_n(t)\right|\leq K_0 \left|\beta_n(0)\right|e^{\omega^+ t}\\
\label{bootstrap_domination_beta}
\forall t\in [0,t_n), &\quad \|\partial_{t,x}h_n(t)\|_2 \leq K_0|\beta_n(t)|.
\end{align}
\end{lemma}
Before proving the lemma, we will show that it implies Proposition \ref{P:exittime}. For this we take $M=1+\max\{M_0,K_0\}$ and apply Lemma \ref{L:bootstrap*}. Then by \eqref{bootstrap_domination_beta}, $\|\partial_{t,x}h_n(t)\|_2<M|\beta_n(t)|$ on $[0,t_n]$ and thus
$$ t_n(\eta,M)=\inf\big\{t\geq 0 \;:\;\,|\beta_n(t)|\geq \eta\big\}=T_n(\eta).$$
This shows by \eqref{bootstrap_beta} that $T_n(\eta)\in (0,+\infty)$ for large $n$, and by continuity of 
$\beta_n$, that $\beta_n(T_n(\eta))=\eta$. In particular, $T_n(\eta)$ must tend to infinity; otherwise, as $\beta_n(0)$ tends to $0$, the continuity of the flow would imply that $u_n(T_n(\eta))$ tends to $W$ and $\beta_n(T_n(\eta))$ to $0$, a contradiction. By \eqref{bootstrap_beta}, we get for large $n$,
$$ \frac{1}{K_0}\left|\beta_n(0)\right|e^{\omega^-T_n(\eta)}\leq \eta \leq K_0 \left|\beta_n(0)\right|e^{\omega^+T_n(\eta)}.$$
By the upper bound inequality we get (noticing that $\log \left|\beta_n(0)\right|$ is negative for large time)
$$ \big|\log |\beta_n(0)|\big|+\log \eta \leq \log(K_0)+ \omega^+ T_n(\eta).$$
Hence, using that $\beta_n(0)$ tends to $0$, as $n$ goes to infinity,
$$ \frac{1}{\omega^+}\leq \liminf_{n\rightarrow +\infty} \frac{T_n(\eta)}{\big|\log |\beta_n(0)|\big|}.$$
Letting $\omega_+$ tends to $\omega$ we get
$$ \frac{1}{\omega}\leq \liminf_{n\rightarrow +\infty} \frac{T_n(\eta)}{\big|\log |\beta_n(0)|\big|}.$$
By the same argument, we get 
$$\limsup_{n\rightarrow +\infty} \frac{T_n(\eta)}{\big|\log |\beta_n(0)|\big|}\leq \frac{1}{\omega},$$ 
which concludes the proof of \eqref{estimate_time}.

To conclude the proof of Proposition \ref{P:exittime} observe that \eqref{growth_beta} implies, for large $n$, 
$$\omega^-\eta=\omega^-|\beta_n(T_n(\eta))|\leq |\beta'_n(T_n(\eta))|,$$
which yields \eqref{estimate_beta'}.

In the remainder of this subsection we prove Lemma \ref{L:bootstrap*}.

\begin{proof}[Proof of Lemma \ref{L:bootstrap*}]
In view of Claim \ref{C:bound.eps}, the fact that $\beta_n(0)$ tends to $0$ and the continuity of $\beta_n$ and $\|\partial_{t,x}h_n\|_2$, the time $t_n$ is strictly positive. Furthermore,
\begin{align}
\label{beta_small}
\forall n,\;\forall t\in (0,t_n), &\quad |\beta_n(t)|\leq \eta\\
\label{domination_beta}
\forall n,\;\forall t\in (0,t_n), &\quad\|\partial_{t,x}h_n(t)\|_2\leq M\left|\beta_n(t)\right|.
\end{align}

%We must show that if $\eta\leq \frac{1}{K_0M^3}$, $K_0$ and $\tau_0$ are large \eqref{growth_beta}, \eqref{bootstrap_beta} and \eqref{bootstrap_domination_beta} hold.

\EMPH{Proof of \eqref{growth_beta}}

Let $$m=\frac{1}{2}\min\left\{\omega^2-(\omega^-)^2,(\omega^+)^2-\omega^2\right\}.$$
We first show that if $\eta$ satisfies \eqref{bound_eta}, then
\begin{equation}
\label{diff_beta3}
\left|{\beta_n}''-\omega^{2}\beta_n\right|\leq m\left|\beta_n(t)\right|.
\end{equation} 

Differentiating twice the equality $\beta_n=\int h_n\YYY$, we get, by equation \eqref{linearized},
$$ {\beta_n}''-\omega^{2}\beta_n=\int \partial_t^2 h_n\,\YYY-\omega^{2}\int h_n\YYY=\int R(h_n)\YYY-\int \left(Lh_n+\omega^{2}h_n\right)\YYY=\int R(h_n)\YYY.$$
Thus there exists a constant $C_1$, independent of all parameters, such that
\begin{equation}
\label{diff_beta}
\left|{\beta_n}''-\omega^{2}\beta_n\right|\leq C_1\|\partial_{t,x}h_n\|_2^2.
\end{equation} 
By \eqref{beta_small} and \eqref{domination_beta}
$$ \left|{\beta_n}''-\omega^{2}\beta_n\right|\leq C_1M^2\beta_n^2\leq C_1M^2 \eta |\beta_n|.$$
which yields, if $C_1M^2\eta\leq m$ (which follows from \eqref{bound_eta} if $K_0$ is large enough), the desired estimate \eqref{diff_beta3}. 

\medskip

In what follows, we will assume that $\beta_n(0)\geq 0$ and ${\beta_n}'(0)\geq 0$ (otherwise, replace $\beta_n$ by $-\beta_n$ in the forthcoming argument). We next show that for $t\in(0,t_n)$,
\begin{equation}
\label{positivity}
{\beta_n}''(t)>0,\quad {\beta_n}'(t)>0,\quad {\beta_n}(t)>0.
\end{equation}
Indeed by \eqref{diff_beta3},
\begin{equation}
\label{ineq_beta}
(\omega^2-m)\beta_n(t)\leq {\beta_n}''(t)\leq (\omega^2+m)\beta_n(t).
\end{equation}
As $\beta_n(0)>0$ by Claim \ref{C:bound.eps}, we get that $\beta$, $\beta'$ and $\beta''$ are (strictly) positive for small positive $t$. This shows that \eqref{positivity} holds near $0$, and by an elementary monotonicity argument, that it holds for all $t\in (0,t_n]$. 

\medskip

We are now ready to show \eqref{growth_beta}. For this we write, as a consequence of \eqref{ineq_beta} 
$$ \left(\beta'_n-\omega^-\beta_n\right)'=\beta''_n-\omega^-\beta_n'\geq -\omega^-\left(\beta'_n-\omega^-\beta_n\right)+(\omega^2-(\omega^-)^2-m)\beta_n\geq -\omega^-\left(\beta'_n-\omega^-\beta_n\right)+m\beta_n.$$
Hence (using that $\beta_n$ increases with time)
$$ \frac{d}{dt}\left[e^{\omega^- t} \left(\beta_n'-\omega^-\beta_n\right)\right]\geq m e^{\omega^- t} \beta_n(0).$$
Integrating between $0$ and $t$ we get
$$ e^{\omega^-t} \left(\beta_n'-\omega^-\beta_n\right)\geq m\beta_n(0)\int_0^t e^{\omega^-s} ds+\beta_n'(0)-\omega^-\beta_n(0)\geq \beta_n(0)\left[m\frac{e^{\omega^- t}-1}{\omega^-}-\omega^- \right].$$
Chosing $\tau_0$ large enough we get a positive right hand side for $t\geq \tau_0$, hence the left inequality in \eqref{growth_beta}. The right inequality follows similarly by differentiating $\beta'_n-\omega^+\beta_n$ and we omit the details of the proof.

\EMPH{Proof of \eqref{bootstrap_beta}}
Assume as in the proof of \eqref{growth_beta} that $\beta_n(0)\geq 0$ and $\beta_n'(0)\geq 0$. By \eqref{ineq_beta}, and using that $\beta_n$ is positive on $(0,T)$,
$$\forall t\in [0,t_n],\quad \beta_n''(t)-{\omega^+}^2\beta_n(t)\leq \left(\omega^2+m-(\omega^+)^2\right)\beta_n(t)<0.$$
This shows by a standard ODE argument that $\beta_n(t)\leq \tilde{\beta}_n(t)$, where $\tilde{\beta}_n(t)$ is the solution of the differential equation $\tilde{\beta}_n''-{\omega^+}^2 \tilde{\beta}_n=0$ with initial conditions $\tilde{\beta}_n(0)=\beta_n(0)$, $\tilde{\beta}_n'(0)=\beta_n'(0)$. Hence
\begin{equation}
\label{other_bound_beta}
\forall t\in [0,t_n],\quad \beta_n(t)\leq \beta_n(0) \cosh(\omega^+ t)+\frac{\beta_n'(0)}{\omega^+}\sinh(\omega^+ t).
\end{equation} 
By \eqref{devh},
$$ \partial_t h_n(0)=\beta_n'(0)\YYY+\sum_{j=0}^N \gamma_{j,n}'(0)W_j +g_n(0),\quad g_n(0)\in G_{\bot}.$$
Taking the $L^2$-scalar product with $\YYY$ and recalling that $W_j$, $j=0\ldots N$, and $g_n(0)$ are orthogonal to $\YYY$, we get $|\beta'_n(0)|\leq \|\partial_t h_n(0)\|_2$. Thus, in view of Claim \ref{C:bound.eps}, for large $n$:
$$ \beta'_n(0)\leq (M_0+1)\beta_n(0).$$
By \eqref{other_bound_beta}
\begin{equation}
\label{bound_beta_tau_0}
\beta_n(\tau_0)\leq \beta_n(0) \cosh(\omega^+ \tau_0)+\frac{M_0+1}{\omega^+}\beta_n(0)\sinh(\omega^+ \tau_0)\leq K_1\beta_n(0),
\end{equation} 
for some constant $K_1$ depending only on the choice of $\omega^+$. By \eqref{growth_beta},
$$ \forall t\geq \tau_0, \quad e^{\omega^- (t-\tau_0)}\beta_n(\tau_0)\leq \beta_n(t)\leq e^{\omega^+ (t-\tau_0)}\beta_n(\tau_0).$$
Using \eqref{bound_beta_tau_0} for the upper bound and the fact that $\beta_n$ increases for the lower bound , we get
$$ \forall t\geq \tau_0, \quad e^{\omega^- (t-\tau_0)}\beta_n(0)\leq \beta_n(t)\leq K_1 e^{\omega^+ (t-\tau_0)}\beta_n(0),$$
which yields \eqref{bootstrap_beta}.

\EMPH{Proof of \eqref{bootstrap_domination_beta}}

We divide the proof into two steps.

\medskip

\noindent\emph{Step 1. Estimates on the coefficients}

We first show that there exist a constant $C_1>0$, independent of the parameters $M$ and $\eta$, such that for all $t\in \left[0,t_n\right]$
\begin{gather}
\label{estimate_energy}
\frac{1}{C_1}|\beta_n| -C_1\|\partial_{t,x} h_n\|_2^{3/2}\leq \left\|\nabla g_n\right\|_2+\left\|\partial_t h_n\right\|_2+\eps_n\leq C_1\beta_n+C_1\|\partial_{t,x}h_n\|_2^{3/2}\\
\label{estimate_compactness}
\frac{1}{C_1}\|\partial_{t,x}h_n\|_2\leq \left|\beta_n\right|+\sum_{j=0}^N \left|\gamma_{j,n}\right|\leq C_1\|\partial_{t,x}h_n\|_2.
\end{gather}

We have
$$ E\left(W+h_n,\partial_t h_n\right)=E(W,0)-\eps_n^2$$ 
Thus there exists a constant $C_2>0$ (independent of the parameters) such that
$$ \left|Q(h_n)+\int \left|\partial_t h_n\right|^2+\eps_n^2\right|\leq C_2\|\partial_{t,x}h_n(t)\|_2^{3}.$$
Furthermore, by \eqref{devh} (and the fact that the functions $W_j$, $j=0\ldots N$ are in the kernel of $Q$)
$$ Q(h_n)=-{\beta_n}^2|Q(\YYY)|+Q\left(g_n\right).$$
Which yields
\begin{equation}
\label{eqEnergy}
\left|-{\beta_n}^2|Q(\YYY)|+Q\left(g_n\right)+\int \left|\partial_t h_n\right|^2+\eps_n^2\right|\leq C_2\|\partial_{t,x}h_n(t)\|_2^{3}.
\end{equation} 
As $g_n\in G_{\bot}$, we have $Q\left(g_n\right)\approx \left\|\nabla g_n\right\|_2^2$, which yields \eqref{estimate_energy}.

Let us show \eqref{estimate_compactness}. Note that the upper bound follows immediately from the definitions of $\beta_n$ and $\gamma_{j,n}$ (see \eqref{beta} and \eqref{gamma}). It remains to show the lower bound. We have
\begin{equation*}
h_n(t)=\beta_n(t)\YYY+\sum_{j=0}^N \gamma_{j,n}(t)W_j +g_n(t),
\end{equation*} 
and hence, by \eqref{estimate_energy}
\begin{gather*}
\left\|\nabla h_n\right\|_2\leq C\left[\left|\beta_n\right|+\sum_{j=0}^N \left|\gamma_{j,n}\right|+\left\|\nabla g_n\right\|_2\right]\leq C \left[\left|\beta_n\right|+\sum_{j=0}^N \left|\gamma_{j,n}\right|+\|\partial_{t,x}h_n\|_2^{\frac 32}\right]\\
\left\|\partial_{t,x}h_n\right\|_2=\left\|\nabla h_n\right\|_2+\left\|\partial_t h_n\right\|_2\leq C\left[\left|\beta_n\right|+\sum_{j=0}^N \left|\gamma_{j,n}\right|+\|\partial_{t,x}h_n\|_2^{\frac 32}\right].
\end{gather*}
As a consequence of \eqref{beta_small} and \eqref{domination_beta}, we obtain
\begin{equation*}
\left\|\partial_{t,x}h_n\right\|_2=\leq C\left[\left|\beta_n\right|+\sum_{j=0}^N \left|\gamma_{j,n}\right|\right]+C\|\partial_{t,x}h_n\|_2M^{1/2}\eta^{1/2}.
\end{equation*}
by \eqref{bound_eta}, we get the lower bound in \eqref{estimate_compactness}

\medskip

\noindent\emph{Step 2. Bound on $\gamma_{j,n}$.}

We are now ready to show \eqref{bootstrap_domination_beta}. According to \eqref{estimate_compactness}, it is sufficient to show that there exists a constant $C_3$ independent of $M$ and $\eta\leq \frac{1}{K_0M^3}$ such that
\begin{equation}
\label{bound_gamma}
\forall j\in \{0,\ldots,N\},\; \forall t\in \left[0,t_n\right],\quad \left|\gamma_{j,n}(t)\right|\leq C_3 \left|\beta_n(t)\right|.
\end{equation} 

We have, for $j=0\ldots N$.
$$ \gamma_{j,n}'(t)=\int \nabla \left(\partial_th_n(t)-{\beta_n}'(t)\YYY\right)\nabla W_j.$$
Note that $\int \nabla W_j \nabla \YYY=0$ if $j\geq 1$, but we won't need this fact in the sequel. The preceding inequality yields
\begin{equation}
\label{bound_gammaj}
\left|\gamma_{j,n}'(t)\right|\leq C\left(\|\partial_t h_n(t)\|_2+\left|{\beta_n}'(t)\right| \right).
\end{equation} 
By \eqref{estimate_energy} and assumptions \eqref{beta_small} and \eqref{domination_beta},
\begin{equation*}
\|\partial_t h_n\|_2\leq C_1\left(|\beta_n|+\|\partial_{t,x}h_n\|_2^{3/2}\right)\leq C_1|\beta_n|\left(1+\eta^{1/2} M^{3/2}\right).
\end{equation*}
Taking $\eta$ small enough so that $\eta^{1/2} M^{3/2}\leq 1$, we get
\begin{equation}
\label{bound_dth}
\|\partial_t h_n\|_2\leq 2C_1|\beta_n|.
\end{equation}  
By \eqref{bound_gammaj}, taking a larger constant $C$,
\begin{equation}
\label{bound_gammaj2}
\left|\gamma_{j,n}'\right|\leq C\left(\left|\beta_n\right|+\left|\beta_n'\right| \right).
\end{equation} 
Integrating between $0$ and $t\leq \tau_0$, and using that $\gamma_{j,n}(0)=0$, that $|\beta_n|$ increases and that the sign of $\beta'_n(t)$ is independant of $t\in [0,t_n]$ (see \eqref{positivity}), we obtain
\begin{equation}
\label{bound_small_t}
\forall t\in \left[0,\tau_0\right],\quad |\gamma_{j,n}(t)|\leq C(t+1)|\beta_n(t)|.
\end{equation}
This yields \eqref{bound_gamma} for $t\leq \tau_0$.

Now by \eqref{growth_beta} and \eqref{bound_gammaj2}, and using that the signs of $\beta_n$ and $\beta_n'$ do not depend on time,
\begin{equation*}
\forall t\geq \tau_0,\quad \left|\gamma_{j,n}'(t)\right|\leq C\left|{\beta_n}'(t)\right|.
\end{equation*} 
Integrating between $\tau_0$ and $t\in \left[\tau_0,t_n\right]$, we get
\begin{equation*}
\left|\gamma_{j,n}(t)\right|\leq C\left(\left|\beta_n(t)\right|+\left|\gamma_{j,n}(\tau_0)\right|\right).
\end{equation*} 
Using \eqref{bound_small_t} at $t=\tau_0$ and the fact that $|\beta_n|$ increases, we get \eqref{bound_gamma} for $t\geq \tau_0$. The proof is complete.
\end{proof}

\section{Proof of main result}
\label{S:proof}
This section is devoted to the proof of Theorem \ref{maintheo}. The proof is divided into 3 steps. In Step 1, we show the lower bound, in the next two steps the upper bound.

\EMPH{Step 1. Lower bound}

We must show

\begin{equation}
\label{lower_bound}
\liminf_{\eps\rightarrow 0^+} \frac{\III_{\eps}}{\left|\log \eps\right|}\geq \frac{2}{\omega}\int_{\RR^N} W^{\frac{2(N+1)}{N-2}}.
\end{equation} 

For this we first note that 
$$ \int\nabla W\cdot \nabla \YYY=-\int \Delta W\YYY=\int W^{\frac{N+2}{N-2}}\YYY>0,$$
as $\YYY$ and $W$ are positive. Consider the family of solutions $(u^a)_{a>0}$ of \eqref{CP} with initial conditions
$$ u^a_0=W-a \YYY,\quad u^a_1=0.$$
For small $a>0$,
$$ \int |\nabla u^a_0|^2=\int |\nabla W|^2-2a\int \nabla W\cdot\nabla \YYY+a^2\int |\nabla \YYY|^2<\int |\nabla W|^2.$$
We have
\begin{equation}
\label{energy_ua}
E\left(u^a_0,u^a_1\right)=E(W,0)+Q\left(-a\YYY \right)+O\left(a^3\right)=E(W,0)-a^2\left|Q\left(\YYY \right)\right|+O\left(a^3\right).
\end{equation}
We argue by contradiction. If \eqref{lower_bound} does not hold, there exists a sequence $\eps_n$ which tends to $0$ such that for some $\rho>\omega$
\begin{equation}
\label{absurd_lower_bound}
\forall n,\quad \frac{2}{\rho}\int_{\RR^N} W^{\frac{2(N+1)}{N-2}}\geq \frac{\III_{\eps_n}}{\left|\log \eps_n\right|}.
\end{equation} 
By \eqref{energy_ua}, and using that $E\left(u^a_0,u^a_1\right)$ is a continuous function of $a$, there exists a sequence $a_n$ such that 
$$\eps_n^2=E(W,0)-E\left(u^{a_n}_0,u^{a_n}_1\right),$$
Furthermore,
\begin{equation}
\label{epsnsim}
\eps_n \sim a_n\sqrt{|Q(\YYY)|} \text{ as } n\rightarrow +\infty.
\end{equation} 
Let $u_n=u^{a_n}$. Observe that 
$$ \partial_tu_{n}(0)=0,\quad \|\nabla(u_{n}(0)-W)\|_2=a_n \|\nabla \YYY\|_2\underset{n\rightarrow +\infty}{\longrightarrow} 0.$$
Furthermore, $\beta_n(0)=-a_n$, $\beta_n'(0)=0$, which shows that the assumptions of Proposition \ref{P:exittime} are satisfied. Consider a small $\eta>0$. By Proposition \ref{P:exittime}
\begin{equation*}
\lim_{n\rightarrow +\infty} \frac{T_n(\eta)}{|\log a_n|}=\lim_{n\rightarrow +\infty} \frac{T_n(\eta)}{\big|\log |\beta_n(0)|\big|}=\frac{1}{\omega} 
\end{equation*}
By \eqref{epsnsim},
\begin{equation}
\label{estimate_Tn}
\lim_{n\rightarrow +\infty} \frac{T_n(\eta)}{\big|\log |\eps_n|\big|}=\frac{1}{\omega}.
\end{equation}

Let us give a lower bound for $\|u_n\|_{S(0,+\infty)}$. From now on we will write $T_n$ instead of $T_n(\eta)$ for the sake of simplicity. As $u_n=W+h_n$, we have
$$ \|u_n\|_{S(0,T_n)}\geq \|W\|_{S(0,T_n)}-\|h_n\|_{S(0,T_n)}.$$
Furthermore,
$$\|W\|_{S(0,T_n)}=  T_n^{\frac{N-2}{2(N+1)}}\|W\|_{\frac{2(N+1)}{N-2}}.$$
Write
$$\|h\|_{S(0,T_n)}^{\frac{2(N+1)}{N-2}}=\sum_{I\in E_{T_n}} \|h\|_{S(I)}^{\frac{2(N+1)}{N-2}},$$
where $E_{T_n}$ is a set of at most $\frac{T_n}{\tilde{c}}+1$ subinterval of $(0,T_n)$, of length at most $\tilde{c}$ (given by Lemma \ref{L:strichartz}) such that $(0,T_n)=\bigcup_{I\in E_{T_n}} \overline{I}$. By Lemma \ref{L:strichartz} and the fact that $\|\nabla h_n\|_2+\|\partial_t h_n\|_2\leq M\eta$ on $(0,T_n)$, we get for small $\eta>0$,
$$ \|h_n\|_{S(0,T_n)}^{\frac{2(N+1)}{N-2}}\leq C\left(\frac{T_n}{\tilde{c}}+1\right)\eta^{\frac{2(N+1)}{N-2}}.$$
Hence a constant $C>0$ such that 
$$ \|h_n\|_{S(0,T_n)}\leq C \eta \, T_n^{\frac{N-2}{2(N+1)}}.$$
Combining the preceding estimates, we obtain
$$ \int_0^{T_n} \int_{\RR^N}|u_n|^{\frac{2(N+1)}{N-2}} \geq T_n\left[\|W\|_{\frac{2(N+1)}{N-2}}-C\eta\right]^{\frac{2(N+1)}{N-2}}.$$
Hence with \eqref{estimate_Tn},
\begin{equation*}
\liminf_{n\rightarrow +\infty}\frac{1}{|\log \eps_n|}\int_0^{+\infty} \int_{\RR^N}|u_n|^{\frac{2(N+1)}{N-2}} 
\geq \frac{1}{\omega}\left[\|W\|_{\frac{2(N+1)}{N-2}}-C\eta\right]^{\frac{2(N+1)}{N-2}}.
\end{equation*} 
Letting $\eta$ tends to $0$ we obtain 
$$ \liminf_{n\rightarrow+\infty} \frac{1}{|\log \eps_n |} \int_0^{+\infty} |u_n|^{\frac{2(N+1)}{N-2}}\geq \frac{1}{\omega}\|W\|_{\frac{2(N+1)}{N-2}}^{\frac{2(N+1)}{N-2}}.$$
Next, notice that as $\partial_t u(0)=0$, the uniqueness in the Cauchy problem \eqref{CP} implies $u(t,x)=u(-t,x)$ and thus
$$ \liminf_{n\rightarrow+\infty} \frac{1}{|\log \eps_n |} \int_{-\infty}^{0} |u_n|^{\frac{2(N+1)}{N-2}}\geq \frac{1}{\omega}\|W\|_{\frac{2(N+1)}{N-2}}^{\frac{2(N+1)}{N-2}}.$$
Finally,
$$ \liminf_{n\rightarrow+\infty} \frac{\III_{\eps_n}}{|\log \eps_n |} \geq \liminf_{n\rightarrow+\infty} \frac{1}{|\log \eps_n |} \int_{-\infty}^{+\infty} |u_n|^{\frac{2(N+1)}{N-2}}\geq \frac{2}{\omega}\|W\|_{\frac{2(N+1)}{N-2}}^{\frac{2(N+1)}{N-2}},$$
contradicting \eqref{absurd_lower_bound}. Step 1 is complete.

\EMPH{Step 2. Estimate before the exit time}

We next show the upper bound on $\III_{\eps}$, i.e that 
\begin{equation}
\label{upper_bound}
\limsup_{\eps\rightarrow 0^+} \frac{\III_{\eps}}{\left|\log \eps\right|}\leq \frac{2}{\omega}\int_{\RR^N} W^{\frac{2(N+1)}{N-2}}.
\end{equation}

For this we will show that if $\eps_n>0$ is a sequence that goes to $0$ and $u_n$ a sequence of solutions of \eqref{CP} such that
\begin{equation}
\label{hyp.un2.bis}
\|\nabla u_n(0)\|_2<\|\nabla W\|_2,\quad E(W,0)-E(u_n,\partial_t u_n)=\eps_n^2,
\end{equation}
then 
\begin{equation}
\label{CV}
\limsup_{n\rightarrow +\infty} \frac{1}{\left|\log \eps_n\right|}\int_{\RR\times\RR^N} |u_n|^{\frac{2(N+1)}{N-2}}\leq \frac{2}{\omega}\int_{\RR^N} W^{\frac{2(N+1)}{N-2}}.
\end{equation}
Possibly time-translating $u_n$, we may assume
\begin{equation}
\label{split_S}
\left\|u_{n}\right\|_{S(-\infty,0)}=\left\|u_{n}\right\|_{S(0,+\infty)}\underset{n\rightarrow +\infty}{\longrightarrow} +\infty
\end{equation}
By Proposition \ref{P:toW}, rescaling and space-translating $u_n$ if necessary, we can assume
$$ \lim_{n\rightarrow+\infty} u_n=W.$$
Consider the functions $h_n$ and $g_n$, and the parameters $\beta_n$ and $\gamma_{j,n}$ defined in the beginning of \S \ref{SS:bootstrap}.
Replacing $u_n(x,t)$ by $u_n(x,-t)$ if it is not the case, we may assume
\begin{equation}
\label{direction_beta}
\beta_n(0){\beta_n}'(0)\geq 0.
\end{equation} 

Furthermore, by Claim \ref{C:orthogonality}, we may also assume \eqref{orthogonality1}.

Fix a small $\eta>0$, and consider $T_n=T_n(\eta)$ defined by Proposition \ref{P:exittime}. In this step, we show that there exists a constant $C>0$ such that
\begin{equation}
\label{BoundS2}
\limsup_{n\rightarrow +\infty} \frac{1}{|\log \eps_n|}\int_0^{T_n(\eta)} \int_{\RR^N} |u_n|^{\frac{2(N+1)}{N-2}}\leq \frac{1}{\omega}\left[\|W\|_{\frac{2(N+1)}{N-2}}+C\eta\right]^{\frac{2(N+1)}{N-2}}.
\end{equation} 
Indeed, by Claim \ref{C:bound.eps}, for large $n$,
\begin{equation*} 
\eps_n\leq M_0|\beta_n(0)|.
\end{equation*} 
Hence by Proposition \ref{P:exittime}, 
\begin{equation}
\label{estimate_eps}
\limsup_{n\rightarrow+\infty} \frac{T_n}{|\log \eps_n|}\leq \frac{1}{\omega}.
\end{equation}
By the same argument as in Step 1, we get
$$ \int_0^{T_n} \int_{\RR^N}|u_n|^{\frac{2(N+1)}{N-2}} \leq T_n\left[\|W\|_{\frac{2(N+1)}{N-2}}+C\eta\right]^{\frac{2(N+1)}{N-2}}.$$
Hence 
$$\limsup_{n\rightarrow+\infty} \frac{1}{T_n}\int_0^{T_n} \int_{\RR^N}|u_n|^{\frac{2(N+1)}{N-2}} \leq
\left[\|W\|_{\frac{2(N+1)}{N-2}}+C\eta\right]^{\frac{2(N+1)}{N-2}}.$$
Combining with \eqref{estimate_eps}, we obtain \eqref{BoundS2}. 

\EMPH{Step 3. Estimate for large time}

To conclude the proof, we will show that if $\eta$ is small enough, there exists a constant $C(\eta)>0$ such that for large $n$
\begin{equation}
\label{bound_S} 
\|u_n\|_{S(T_n(\eta),+\infty)}\leq C(\eta). 
\end{equation} 
Assuming \eqref{bound_S}, we obtain by \eqref{BoundS2},
$$ \limsup_{n\rightarrow+\infty} \frac{1}{|\log \eps_n|}\|u_n\|_{S(0,+\infty)}^{\frac{2(N+1)}{N-2}}=\limsup_{n\rightarrow+\infty} \frac{1}{|\log \eps_n|}\|u_n\|_{S(0,T_n(\eta))}^{\frac{2(N+1)}{N-2}}\leq \frac{1}{\omega}\left[\|W\|_{\frac{2(N+1)}{N-2}}+C\eta\right]^{\frac{2(N+1)}{N-2}}.
$$
Letting $\eta$ tend to $0$ we get
$$\limsup_{n\rightarrow+\infty} \frac{1}{|\log \eps_n|}\|u_n\|_{S(0,+\infty)}^{\frac{2(N+1)}{N-2}}\leq \frac{1}{\omega}\|W\|_{\frac{2(N+1)}{N-2}}^{\frac{2(N+1)}{N-2}},$$
which shows, in view of \eqref{split_S}, the desired estimate \eqref{CV}.

It remains to show \eqref{bound_S}. We will argue by contradiction. If \eqref{bound_S} does not hold, there exist a subsequence of $(u_n)$, still denoted by $(u_n)$ such that
\begin{equation}
\label{inftyS+}
\|u_n\|_{S(T_{n},+\infty)} \underset{n\rightarrow\infty}{\longrightarrow} +\infty.
\end{equation}
Furthermore, by \eqref{split_S} 
\begin{equation}
\label{inftyS-} 
\|u_{n}\|_{S(-\infty,T_n)}\geq \|u_{n}\|_{S(-\infty,0)}=\|u_{n}\|_{S(0,+\infty)}\underset{n\rightarrow\infty}{\longrightarrow} +\infty.
\end{equation} 
In view of \eqref{inftyS+} and \eqref{inftyS-}, Proposition \ref{P:toW} \eqref{C:toW} implies that there exists sequences $\lambda_n>0$, $x_n\in \RR^N$, and $\delta_0\in \{-1,+1\}$ such that
\begin{equation}
\label{uepstoWbis} 
\lim_{n\rightarrow +\infty} \left\|\frac{\delta_0}{\lambda_n^{N/2}}\nabla u_n\left(T_{n},\frac{\cdot-x_n}{\lambda_n}\right)-\nabla W\right\|_{2}
+\left\|\frac{\partial u_n}{\partial t}(T_{n})\right\|_{2}=0.
\end{equation}
 By Proposition \ref{P:exittime}, 
$$ \liminf_{n\rightarrow+\infty}|\beta_n'(T_n)|\geq \omega \eta.$$
By the decomposition \eqref{devh} of $h_n$,
$$\int \partial_t u_n(T_n)\YYY=\int \partial_t h_n(T_n) \YYY=\beta_n'(T_n).$$
This shows by \eqref{uepstoWbis} that $\beta_n'(T_n)$ must tend to $0$, yielding a contradiction. This concludes the proof of \eqref{bound_S} and thus of Theorem \ref{maintheo}.

\section{Estimate of the scattering norm for energy-critical focusing NLS}
\label{S:NLS}
In this section we briefly adress the case of the radial energy critical focusing semilinear Schr\"odinger equation
\begin{equation}
\label{NLS}
i\partial_t u+\Delta u+|u|^{\frac{4}{N-2}}u=0,\quad u_{\restriction t=0}=u_0\in \hdot_{r},
\end{equation} 
where $N\in\{3,4,5\}$ and $\hdot_r$ is the subset of $\RR^N$ of spherically symmetric functions. The equation \eqref{NLS} is locally well-posed (see \cite{CaWe90}) in the energy space $\hdot_r$. Furthermore, if $I_{\max}\ni 0$ is the maximal interval of definition then
$$J\Subset I_{\max}\Longrightarrow \|u\|_{\tS(J)}<\infty,\text{ where }\tS(J)=L^{\frac{2(N+2)}{N-2}},$$
and globally defined solutions of \eqref{NLS} such that $\|u\|_{\tS(\RR)}$ is finite scatter (see \cite{Bo99JA,Bo99BO}).

The energy 
$$\EEE(u(t))=\frac 12\int |\nabla u(t)|^2-\frac{N-2}{2N}\int |u(t)|^{\frac{2N}{N-2}}$$
is conserved.

In the defocusing case, all solutions are known to be globally defined and scatter \cite{Bo99JA,Ta05NY}. Furthermore, in \cite{Ta05NY}, T.~Tao gave a bound of $\|u\|_{\tS(\RR)}$ in term of an exponential of a power of the conserved defocusing energy $\frac 12\int |\nabla u_0|^2+\frac{N-2}{2N}\int |u|^{\frac{2N}{N-2}}$.

In the focusing case, the function $W$, defined in \eqref{defW} is still a stationnary solution of $W$. 
The following theorem shown in  \cite{KeMe06} for the case $\EEE(u_0)<\EEE(W)$ and in \cite{DuMe07a} for the case $\EEE(u_0)=\EEE(W)$, is the analoguous of Theorem \ref{T:classification} for equation \eqref{NLS}.
\begin{exttheo}[Kenig-Merle,Duyckaerts-Merle]
\label{T:classification_NLS}
There exists a global solution $\tW_-$ of \eqref{NLS} such that 
\begin{gather*}
\EEE\big(\tW^-\big)=\EEE(W), \quad \big\|\nabla \tW_-(0)\big\|_2<\|\nabla W\|_2\\
\big\|\tW^-\big\|_{\tS(-\infty,0)}<\infty,\quad \lim_{t\rightarrow +\infty} \big\|\nabla\big(\tW^-(t)-W\big)\big\|_2=0.
\end{gather*}
Moreover, if $u$ is a radial solution of \eqref{NLS} such that $\EEE(u_0)\leq \EEE(W)$ and $\|\nabla u_0\|_2\leq \|\nabla W\|_2$, then $u$ is globally defined. If furthermore $\|u\|_{\tS(\RR)}=\infty$, then $u=\tW^-$ or $u=W$ up to the invariances of the equation.
\end{exttheo}
Defining 
$$\tilde{F}_{\eps}:=\left\{u \text{ radial solution of }\eqref{NLS}\text{ such that }\EEE(u_0)\leq \EEE(W)-\eps^2 \text{ and }\int |\nabla u_0|^2<\int |\nabla W|^2\right\}.$$  
we get in particular that for $\eps>0$ the supremum 
$$ \tIII_{\eps}=\sup_{u\in \tilde{F}_{\eps}} \int_{\RR\times \RR^N}|u(t,x)|^{\frac{2(N+2)}{N-2}}\,dtdx=\sup_{u\in \tilde{F}_{\eps}} \|u\|_{\tS(\RR)}^{\frac{2(N+2)}{N-2}},$$
is finite, and that 
$$ \lim_{\eps\rightarrow 0^+} \tIII_{\eps}=+\infty.$$
We wish again to estimate of $\tIII_{\eps}$ when $\eps$ goes to $0$. As in the case of the wave equation, the behavior of $\tIII_{\eps}$ is determined by the linearized operator near $W$. If $u=W+h$ is a solution of \eqref{NLS}, then, identifying $h$ with the column vector $(\re h,\im h)^T=(h_1,h_2)^T$.
\begin{gather*}
\label{equation.v}
\partial_t h+ \LLL(h)+R(h)=0,\quad \LLL:=\begin{pmatrix} 0 & \Delta+W^{\frac{4}{N-2}} \\ -\Delta-\frac{N+2}{N-2}W^{\frac{4}{N-2}} & 0 \end{pmatrix},
\end{gather*}
where an appropriate norm of $R(h)$ is bounded by $\|\nabla h\|_2^2$ when $h$ is small. It is known (see \cite[Section 7.1]{DuMe07a}) that the essential spectrum of $\LLL$ is $i\RR$ and that $\LLL$ admits only two nonzero real eigenvalues, $\tomega>0$ and $-\tomega$, with eigenfunctions $\tYYY_{\pm}$ which are in the space of Schwartz functions. Then:
\begin{theo}
\label{theoNLS}
$$\lim_{\eps\rightarrow 0^+} \frac{\tIII_{\eps}} {|\log \eps|}=\frac{2}{\tomega}\int_{\RR^N} W^{\frac{2(N+2)}{N-2}}.$$
\end{theo}
Our result is restricted to the radial case in spatial dimensions $N\in \{3,4,5\}$. In view of the recent work \cite{KiVi08P} on non-radial energy-critical focusing NLS in dimension $N\geq 5$, it is natural to expect that the same estimate holds in a more general situation.

The proof of Theorem \ref{theoNLS} is very similar to the one of Theorem \ref{maintheo}, and we will only sketch it, highlighting the minor differences. In \S \ref{SS:NLS1} we recall a few facts about the operator $\LLL$ and state without proof the analoguous of Propositions \ref{P:toW}, Propositions \ref{P:exittime} and Claim \ref{C:bound.eps}. In \S \ref{SS:sketchNLS} we briefly explain how to use these results to show Theorem \ref{theoNLS}. 

\subsection{Convergence to $W$ and estimate on the exit time}
\label{SS:NLS1}
In view of Theorem \ref{T:classification_NLS}, and the use of the profile decomposition method in \cite[Section 4]{KeMe06} (see also \cite[Lemma 2.5]{DuMe07a}), the proof of Proposition \ref{P:toWNLS} adapts easily to show:
\begin{prop}
\label{P:toWNLS}
Let $u_n$ be a family of radial solutions of \eqref{NLS}, such that
\begin{equation}
\label{below_thresholdNLS}
\EEE\big(u_n(0)\big)<\EEE(W),\quad \|\nabla u_n(0)\|_2<\|\nabla W\|_2.
\end{equation} 
and $\lim_{n\rightarrow +\infty} \|u_n\|_{\tS(\RR)}=+\infty$. Let
$(t_{n})_n$ be a time sequence.
Assume
$$\lim_{n\rightarrow +\infty} \|u_n\|_{\tS(-\infty,t_n)}=\lim_{n\rightarrow +\infty} \|u_{n}\|_{\tS(t_n,+\infty)}=+\infty.$$
Then, up to the extraction of a subsequence, there exist $\theta_0\in \RR$ and a sequence of parameters $\lambda_n>0$ such that
\begin{equation*}
%\label{uepstoW}
\lim_{n\rightarrow +\infty} 
\left\|\frac{e^{i\theta_0}}{\lambda_n^{N/2}}\nabla u_{n}\left(t_{n},\frac{\cdot}{\lambda_n}\right)-\nabla  W\right\|_{2}=0.
\end{equation*}
\end{prop}
We next recall some spectral properties of the operator $\LLL$. We refer to \cite[\S 5.1]{DuMe07a} for the details. We will often identify a complex-valued function $f$ with an $\RR^2$-valued function $(f_1,f_2)^T$, with $f_1=\re f$, $f_2=\im f$. Developping the energy around $W$, we get, for small functions $h\in \hdot$, 
$$ \EEE(W+h)=\EEE(W)+\tQ(h)+O\Big(\|h\|^3_{\frac{2N}{N-2}}\Big),$$
where $\tQ$ is the quadratic form $\tQ(h)=\BB(h,h)$ and $\BB$ is defined by
$$\BB(g,h)=\frac{1}{2}\int \nabla g_1\cdot\nabla h_1-\frac{N+2}{2(N-2)}\int g_1 h_1 W^{\frac{4}{N-2}}+\frac{1}{2}\int \nabla g_2\cdot\nabla h_2-\frac{1}{2}\int g_2 h_2 W^{\frac{4}{N-2}}.$$
Denote by $\YYY_+$ the eigenfunction of $\LLL$ for the eigenvalue $\tomega$ and $\YYY_-=m\overline{\YYY}_+$ the eigenfunction of $\LLL$ for the eigenvalue $-\tomega$ ($m\neq 0$ is a real normalization constant), and recall the definition of $W_0$ in \eqref{def_Wj}. One may show that $W_0$ and $iW$ are in the kernel of $\tQ$. Furthermore, $\tQ(\YYY_+)=\tQ(\YYY_-)=0$ and we may chose $m$ such that $\BB(\YYY_+,\YYY_-)=-1$. Let
$$ \tG_{\bot}:=\left\{h\in \hdot\;:\;\int \nabla W\cdot\nabla h_2=\int \nabla W_0\cdot \nabla h_1=\BB(\YYY_+,h)=\BB(\YYY_-,h)=0\right\}.$$
By \cite[Lemma 5.2]{DuMe07a}, there exists a constant $c>0$ such that
$$ \forall h\in \tG_{\bot},\quad Q(h)\geq c\|\nabla h\|_2^2.$$
We consider as in \S \ref{SS:bootstrap} a sequence $u_n$ of radial solutions of \eqref{NLS} such that
\begin{gather}
\label{subcritical_NLS}
\EEE(u_n)\leq \EEE(W)-\eps_n^2,\quad \|\nabla u_n(0)\|_2<\|\nabla W\|_2,\\
\label{limite_W_NLS}
\lim_{n\rightarrow +\infty} \|\nabla u_n(0)-\nabla W\|_2=0,
\end{gather} 
and develop $h_n=u_n-W$ as follows
\begin{equation}
\label{decomposition_NLS}
h_n(t)=\beta_n^+(t)\YYY_++\beta_n^-(t)\YYY_-+\gamma_n(t)W_0+\delta_n(t)iW+g_n(t),\quad g_n(t)\in \tG_{\bot}.
\end{equation} 
Arguing as in Claim \ref{C:orthogonality}, we may assume 
\begin{equation}
\label{orthogonalityNLS}
\gamma_n(0)=\delta_n(0)=0
\end{equation} 
Then we have the following analog of Propositions \ref{P:exittime} and Claim \ref{C:bound.eps}. We skip the proofs, that are very similar to the previous ones. 
\begin{prop}
\label{P:exittimeNLS}
There exist a constant $\eta_0$, such that for all $\eta\in (0,\eta_0)$, for all sequence $(u_n)$ satisfying \eqref{subcritical_NLS}, \eqref{limite_W_NLS} and \eqref{orthogonalityNLS} if
\begin{align*}
T_n^+(\eta)&=\inf\big\{t\geq 0 \;:\;\,|\beta_n^-(t)|\geq \eta\big\}\\
T_n^-(\eta)&=-\sup\big\{t\leq 0 \;:\;\,|\beta_n^+(t)|\geq \eta\big\}.
\end{align*}
then for large $n$, $T_n^+(\eta)$ and $T_n^-(\eta)$ are finite and
\begin{equation}
\label{estimate_timeNLS}
\lim_{n\rightarrow +\infty} \frac{T_n^+(\eta)}{\log |\beta_n^-(0)|}=\lim_{n\rightarrow +\infty} \frac{T_n^-(\eta)}{\log |\beta_n^+(0)|}=\frac{1}{\tomega}.
\end{equation}
Furthermore,
\begin{equation}
\label{estimate_derivative_NLS}
\liminf_{n\rightarrow +\infty} |\beta_n'(T_n^+(\eta))|\geq \frac{\eta\tomega}{2},\quad \liminf_{n\rightarrow +\infty} |\beta_n'(T_n^-(\eta))|\geq \frac{\eta\tomega}{2}.
\end{equation} 
\end{prop}
Observe that in contrast with the wave equation case, there are two eigenfunctions, and that we have distinguished between the coefficient $\beta_n^-$ of $\YYY_-$, which tends to grow for positive times, and the one of $\YYY_+$, which plays a similar role for negative times.
\begin{claim}
\label{C:bound.epsNLS}
There exists $M_0>0$ such that for all sequence $(u_n)$ of solutions of \eqref{NLS} satisfying \eqref{subcritical_NLS}, \eqref{limite_W_NLS} and \eqref{orthogonalityNLS} we have
$$\beta_n^+\beta_n^-(0)\neq 0\text{ and }\limsup_{n\rightarrow +\infty} \frac{\eps_n^2}{\left|\beta_n^+(0)\beta_n^-(0)\right|}\leq M_0.
$$
\end{claim}
\subsection{Sketch of the proof of Theorem \ref{theoNLS}.}
\label{SS:sketchNLS}
\quad

\EMPH{Step 1. Lower bound}

We first show
\begin{equation}
\label{lower_boundNLS}
\liminf_{\eps\rightarrow 0^+} \frac{\tIII_{\eps}}{\left|\log \eps\right|}\geq \frac{2}{\tomega}\int_{\RR^N} W^{\frac{2(N+2)}{N-2}}.
\end{equation} 
Multiplying $\YYY_+$ and $\YYY_-$ by $-1$ if necessary, we may assume $\re\int \nabla W\cdot \nabla \YYY_{\pm}>0$. Consider the family of solutions $(u^a)_{a>0}$ of \eqref{NLS} with initial conditions
$$ u^a_0=W-a \YYY_+-a\YYY_-.$$
Then for small $a>0$, $\int |\nabla u^a_0|^2<\int |\nabla W|^2.$ Furthermore
\begin{equation}
\label{energy_uaNLS}
\EEE\left(u^a_0\right)=\EEE(W)-2a^2+O\left(a^3\right).
\end{equation}
We argue by contradiction. If \eqref{lower_boundNLS} does not hold, there exists a sequence $\eps_n$ which tends to $0$ such that for some $\rho>\tomega$
\begin{equation}
\label{absurd_lower_boundNLS}
\forall n,\quad \frac{2}{\rho}\int_{\RR^N} W^{\frac{2(N+2)}{N-2}}\geq \frac{\tIII_{\eps_n}}{\left|\log \eps_n\right|}.
\end{equation} 
We then chose a sequence $a_n$ such that 
\begin{equation}
\label{epsn_NLS}
\eps_n^2=\EEE(W)-\EEE\left(u^{a_n}_0\right),\quad \eps_n^2 \sim 2a_n^2  \text{ as } n\rightarrow +\infty.
\end{equation} 
Let $u_n=u^{a_n}$. Then the assumptions of Proposition \ref{P:exittimeNLS} are satisfied. Consider a small $\eta>0$. By Proposition \ref{P:exittimeNLS}, noting that $\beta_n^+(0)=\beta_n^-(0)=-a_n$, we get
\begin{equation*}
\lim_{n\rightarrow +\infty} \frac{T_n^+(\eta)}{\left|\log a_n\right|}=\lim_{n\rightarrow +\infty} \frac{T_n^-(\eta)}{\big|\log a_n\big|}=\frac{1}{\tomega} 
\end{equation*}
By \eqref{epsn_NLS},
\begin{equation}
\label{estimate_TnNLS}
\lim_{n\rightarrow +\infty} \frac{T_n^-(\eta)}{\big|\log |\eps_n|\big|}=\lim_{n\rightarrow +\infty} \frac{T_n^+(\eta)}{\big|\log |\eps_n|\big|}=\frac{1}{\tomega}.
\end{equation}
Writing $u_n=W+h_n$, and arguing as in Step 1 of Section \ref{S:proof}, we obtain
$$ \int_0^{T_n^+(\eta)} \int_{\RR^N}|u_n|^{\frac{2(N+2)}{N-2}} \geq T_n^+(\eta)\left[\|W\|_{\frac{2(N+2)}{N-2}}-C\eta\right]^{\frac{2(N+2)}{N-2}}.$$
Hence with \eqref{estimate_TnNLS}, and letting $\eta$ tends to $0$,
$$ \liminf_{n\rightarrow+\infty} \frac{1}{|\log \eps_n |} \int_0^{+\infty} |u_n|^{\frac{2(N+2)}{N-2}}\geq \frac{1}{\tomega}\|W\|_{\frac{2(N+2)}{N-2}}^{\frac{2(N+2)}{N-2}}.$$
Arguing similarly for negative time, we obtain
$$ \liminf_{n\rightarrow+\infty} \frac{\tIII_{\eps_n}}{|\log \eps_n |} \geq \liminf_{n\rightarrow+\infty} \frac{1}{|\log \eps_n |} \int_{-\infty}^{+\infty} |u_n|^{\frac{2(N+2)}{N-2}}\geq \frac{2}{\tomega}\|W\|_{\frac{2(N+2)}{N-2}}^{\frac{2(N+2)}{N-2}},$$
contradicting \eqref{absurd_lower_boundNLS}. Step 1 is complete.

\EMPH{Step 2. Upper bound}

To show the upper bound on $\tIII_{\eps}$,  we must show that for any sequence $\eps_n>0$ that goes to $0$ any sequence $u_n$ of solutions of \eqref{NLS} such that
\begin{equation}
\label{hyp.un2NLS}
\|\nabla u_n(0)\|_2<\|\nabla W\|_2,\quad \EEE(W)-\EEE(u_n)=\eps_n^2,
\end{equation}
we have
\begin{equation}
\label{CVNLS}
\limsup_{n\rightarrow +\infty} \frac{1}{\left|\log \eps_n\right|}\int_{\RR\times\RR^N} |u_n|^{\frac{2(N+2)}{N-2}}\leq \frac{2}{\tomega}\int_{\RR^N} W^{\frac{2(N+2)}{N-2}}.
\end{equation}
In view of Proposition \ref{P:toWNLS} and the analoguous of Claim \ref{C:orthogonality}, we may assume that $u_n$ satisfy the assumptions of Proposition \ref{P:exittimeNLS}.

Fix a small $\eta>0$, and consider $T_n^{\pm}(\eta)$ defined by Proposition \ref{P:exittimeNLS}. Then by the same proof than in Step 2 of Section \ref{S:proof}, one can show that there exists a constant $C>0$ such that
\begin{align*}
\limsup_{n\rightarrow +\infty} \frac{1}{\left|\log |\beta_n^-(0)|\right|}\int_0^{T_n^+(\eta)} \int_{\RR^N} |u_n|^{\frac{2(N+2)}{N-2}}\leq \frac{1}{\tomega}\left[\|W\|_{\frac{2(N+2)}{N-2}}+C\eta\right]^{\frac{2(N+2)}{N-2}}\\
\limsup_{n\rightarrow +\infty} \frac{1}{\left|\log |\beta_n^+(0)|\right|}\int_{T_n^-(\eta)}^{0} \int_{\RR^N} |u_n|^{\frac{2(N+2)}{N-2}}\leq \frac{1}{\tomega}\left[\|W\|_{\frac{2(N+2)}{N-2}}+C\eta\right]^{\frac{2(N+2)}{N-2}}.
\end{align*} 
By Claim \ref{C:bound.epsNLS}, for large $n$,
\begin{equation*} 
2\left|\log \eps_n\right|\geq \Big|\log|\beta_n^-(0)|+\log|\beta_n^+(0)|\Big|+o_n(1),
\end{equation*} 
which yields 
\begin{equation}
\label{small_times_NLS}
\limsup_{n\rightarrow +\infty} \int_{-T_n^-(\eta)}^{T_n^+(\eta)} \int_{\RR^N} |u_n|^{\frac{2(N+2)}{N-2}}\leq \frac{2\left|\log \eps_n\right|}{\tomega}\left[\|W\|_{\frac{2(N+2)}{N-2}}+C\eta\right]^{\frac{2(N+2)}{N-2}}\\
\end{equation} 
It remains to show, as in Step 3 of Section \ref{S:proof}, that  if $\eta$ is small enough, there exists a constant $C(\eta)>0$ such that for large $n$
\begin{equation}
\label{bound2_NLS} 
\|u_n\|_{S(-\infty,-T_n^-(\eta))}+\|u_n\|_{S(T_n^+(\eta),+\infty)}\leq C(\eta). 
\end{equation} 
Combining \eqref{small_times_NLS} and \eqref{bound2_NLS} and letting $\eta$ tends to $0$ we would get \eqref{CVNLS}.

To show \eqref{bound2_NLS}, we argue by contradiction. Assume that there exists a subsequence of $(u_n)$, such that (from now on, we will write $T_n^+=T_n^+(\eta)$)
\begin{equation*}
\|u_n\|_{S(T_{n}^+,+\infty)} \underset{n\rightarrow\infty}{\longrightarrow} +\infty.
\end{equation*}
Then by Proposition \ref{P:toWNLS}, there exists $\theta_0\in \RR$ and a sequence $\lambda_n>0$, such that
\begin{equation}
\label{uepstoWbisNLS} 
\lim_{n\rightarrow +\infty} \left\|\frac{e^{i\theta_0}}{\lambda_n^{N/2}}\nabla u_n\left(T_{n}^+,\frac{\cdot}{\lambda_n}\right)-\nabla W\right\|_{2}=0.
\end{equation}
As in Step 3 of Section \ref{S:proof}, we will get a contradiction by showing that $\frac{d\beta_n^-}{dt}(T_n^+(\eta))$ tends to $0$. Unlike the case of the wave equation, the convergence to $0$ of the time derivative of $u$ is not given directly by the compactness argument of Proposition \ref{P:toWNLS}. However, \eqref{uepstoWbisNLS} and the fact that $u_n$ is a solution of \eqref{NLS} which is in $C^{0}(\RR,\hdot)$ shows that
$$ \lim_{n\rightarrow+\infty} \left\|\partial_t u_n(T_n^+)\right\|_{H^{-1}}=0.$$
As $\YYY_+$ is a Schwartz function (see \cite[\S 7.2.2]{DuMe07a}), we get, at the point $t=T_n^+$,
\begin{equation}
\label{BunY0}
\lim_{n\rightarrow+\infty}\frac{d}{dt}B(u_n(t),\YYY_+))=0.
\end{equation}
The condition $g_n(t)\in G_{\bot}$ implies that $\beta_n^-(t)=-B\big(u_n(t)-W,\YYY_+\big)$. Thus \eqref{estimate_derivative_NLS} contradicts \eqref{BunY0}. This shows
\begin{equation*}
\|u_n\|_{S(T_n^+(\eta),+\infty)}\leq C(\eta). 
\end{equation*} 
By a similar argument for negative time, we get \eqref{bound2_NLS}. Combining \eqref{small_times_NLS} and \eqref{bound2_NLS}, we obtain \eqref{CVNLS}, which concludes the sketch of the proof of Theorem \ref{theoNLS}.

\bibliographystyle{alpha} %plain, acm, apalike, unsrt?
\bibliography{toto}

\end{document}